\documentclass[11pt]{amsart}
\usepackage{geometry}              
\usepackage{hyperref}
\usepackage{MnSymbol}
\usepackage{mathrsfs}
\usepackage{amsmath,amsthm}
\usepackage{tikz}
\usetikzlibrary{
  knots,
  hobby,
  decorations.pathreplacing,
  shapes.geometric,
  calc,
  decorations.markings
}
\usepgfmodule{decorations}
\usepackage{float}
\usepackage{braids}
\usepackage{tikz-cd}

\usepackage{wrapfig}
\usepackage{lipsum} 
\usetikzlibrary{matrix}
\newenvironment{psmallmatrix}
  {\left(\begin{smallmatrix}}
  {\end{smallmatrix}\right)}

\theoremstyle{plain}
\theoremstyle{definition}

\numberwithin{equation}{section}

\DeclareMathAlphabet\mathbb{U}{msb}{m}{n}

\theoremstyle{plain}
\theoremstyle{definition}
\newtheorem{Theorem}{Theorem}[section]
\newtheorem{quest}[Theorem]{Question}

\newtheorem{Lemma}[Theorem]{Lemma}
\newtheorem{Proposition}[Theorem]{Proposition}

\begin{document}

\title{The language of geodesics for the discrete Heisenberg group}

\author{Ilya Alekseev}
\author{Ruslan Magdiev}

\address{Laboratory of Modern Algebra and Applications, St. Petersburg State University, 14th Line, 29b, Saint Petersburg, 199178 Russia}
\email{ilyaalekseev@yahoo.com}
\address{Laboratory of Continuous Mathematical Education (School 564 of St. Petersburg), nab. Obvodnogo kanala 143, Saint Petersburg, Russia}
\email{ruslanmagdy@gmail.com}


\begin{abstract}
In this paper, we give a complete description of the language of geodesic words for the discrete Heisenberg group $H(\mathbb{Z})$ with respect to the standard two-element generating set. More precisely, we prove that the only dead end elements in $H(\mathbb{Z})$ are nontrivial elements of the commutator subgroup. We give a description of their geodesic representatives, which are called dead end words. The description is based on a minimal perimeter polyomino concept. Finally, we prove that any geodesic word in $H(\mathbb{Z})$ is a prefix of a dead end word.
\end{abstract}

\maketitle

\section{Introduction}

We study combinatorial group theory aspects of the discrete Heisenberg group 
\begin{align*}
H(\mathbb{Z}) = \left\lbrace \left. \begin{psmallmatrix}1 & n & k \\0 & 1 & m\\ 0 & 0 & 1\end{psmallmatrix} \right\vert n,m,k \in \mathbb{Z}\right\rbrace.
\end{align*}
This group a subgroup of the Heisenberg group $H(\mathbb{R})$ of similar upper-triangular matrices with real coefficients. The latter is a three dimensional Lie subgroup of $GL(3,\mathbb{R})$ with Nil geometry. The discrete Heisenberg group has the standard presentation
\begin{align*}
H(\mathbb{Z}) \cong \langle a,b \mid [[a,b],a] = 1, \ [[a,b],b]=1\rangle.
\end{align*}
with generators $a,b$ that correspond to $\begin{psmallmatrix}1 & 1 & 0 \\0 & 1 & 0\\ 0 & 0 & 1\end{psmallmatrix}$ and $\begin{psmallmatrix}1 & 0 & 0 \\0 & 1 & 1\\ 0 & 0 & 1\end{psmallmatrix}$ respectively.

By the Cayley graph of a group $G$ with a finite generating set $S$ we mean a one-dimensional cell complex ${\rm Cay}(G,S)$, whose vertices are elements of $G$ and in which from each vertex $g \in G$ there is an edge to $gs$ for each $s \in S$. It has a metric space structure $d_{G,S}$ in which the length of any edge is $1$. We define the length of an element as follows: $l(g) = d_{G,S}(1, g)$, where $1$ is the identity element of $G$. A word $w \in F(S)$ is called geodesic if $l([w]) = |w|$, where $|w|\geq 0$ is a length of the word $w$.

\subsection{Known results on $H(\mathbb{Z})$}

One of the first results on $H(\mathbb{Z})$ that uses the geometric interpretation of words on the Cayley graph was obtained in \cite{Shapiro}. More precisely, non-regularity of the language of geodesic words and almost convexity of $H(\mathbb{Z})$ were proved, an explicit formula for the growth of $H(\mathbb{Z})$ was obtained (see also \cite{DS}), and information about end letters of geodesic words was presented. Another method of computing growth series of $H(\mathbb{Z})$ can be found in \cite{ST}. A sketch of the method is given in book \cite{OH}, where one asks for geodesic growth series formula of $H(\mathbb{Z})$, see Office Hour 12, Projects 5,6. An explicit length formula for elements of $H(\mathbb{Z})$ is given in \cite{SB}. Another similar formula and some information about geodesic words in $H(\mathbb{Z})$ can be found in \cite{Juillet}. An asymptotic behavior of the number of geodesics in $H(\mathbb{Z})$ is studied in \cite{VM1}.

A function $f(n)$ is called rational if it can be expressed in the form $f(n) = P_1(n) \lambda_1^n + \ldots + P_m(n) \lambda_m^n$, where $m \geq 1$, $P_i(n)$ are polynomials, and $\lambda_i \in \mathbb{R}$. Recall that the language of geodesics for a group with a generating set $S$ is called rational if the geodesic growth function $\gamma(n) = |\{w \in F(S) \mid l([w]) = |w| = n\}|$ is rational. It is well known that if a language is regular, then it is also rational. There is the 2016 year problem in geodesic growth theory related to the discrete Heisenberg group.

\begin{quest}[List of open questions in \cite{BC}, Question 3]
Is there a group with solvable Word Problem and irrational geodesic growth?
\end{quest}
The group $H(\mathbb{Z})$ is one of the simplest examples of a group with a non-regular language of geodesics. The authors of the question believe that $H(\mathbb{Z})$ is the desired example. Note that on cannot solve the problem above without any insight into the language of geodesics.

A right-infinite word is called geodesic if all its finite parts are geodesic. In \cite{VM} the language of infinite geodesic words in $H(\mathbb{Z})$ was described explicitly. It turned out that in $H(\mathbb{Z})$ there is a finite geodesic word that is not contained in any infinite one. For instance, for every $k \neq 0$ the word $a^kb^ka^{-k}b^{-k}$ is geodesic in $H(\mathbb{Z})$, but it cannot be extended to an infinite geodesic word, and even to a geodesic word one letter longer. In geometric group theory, such words are called dead end words, see \cite{OH}. The elements defined by dead end words are called dead end elements. More precisely, let $G = \langle S \rangle$ be a group with a generating set. An element $g \in G$ is called dead end element (with respect to $S$) if $l(gs) < l(g)$ for all $s \in S \cup S^{-1}$. A geodesic word $w \in F(S)$ is called dead end word if $ws$ is not geodesic for all $s \in S\cup S^{-1}$. For example, in $H(\mathbb{Z})$ elements of the form $[a,b]^k$ with $k \neq 0$ are dead end elements. The notion of dead end elements was introduced by Bogopolski in \cite{BO}. For example, it is known that a Coxeter group with Coxeter generators has dead ends if and only if the group is finite. An explicit description of dead end elements in known for the lamplighter group $\mathbb{Z}\slash 2\mathbb{Z} \wreath \mathbb{Z}$ (see \cite{CTLamp, CET}) and for Thompson group $F$ with a two-element generating set (see \cite{BR,CT}). Information about dead ends in solvable and metabelian groups can be found in \cite{GU,GUS}. Note that dead end elements in $H(\mathbb{Z})$ were studied in \cite{DP} in the context of the group depth notion. More precisely, it was proved that ${\rm depth}(N_2) = \infty$ and that dead ends from the sequence $g_n = [a,b]^{n\cdot n+1}$, where $n>2,$ have depth at least $\sqrt{2n-4}+1$. 
\begin{figure}[H]
\centering
\includegraphics[width = 12cm]{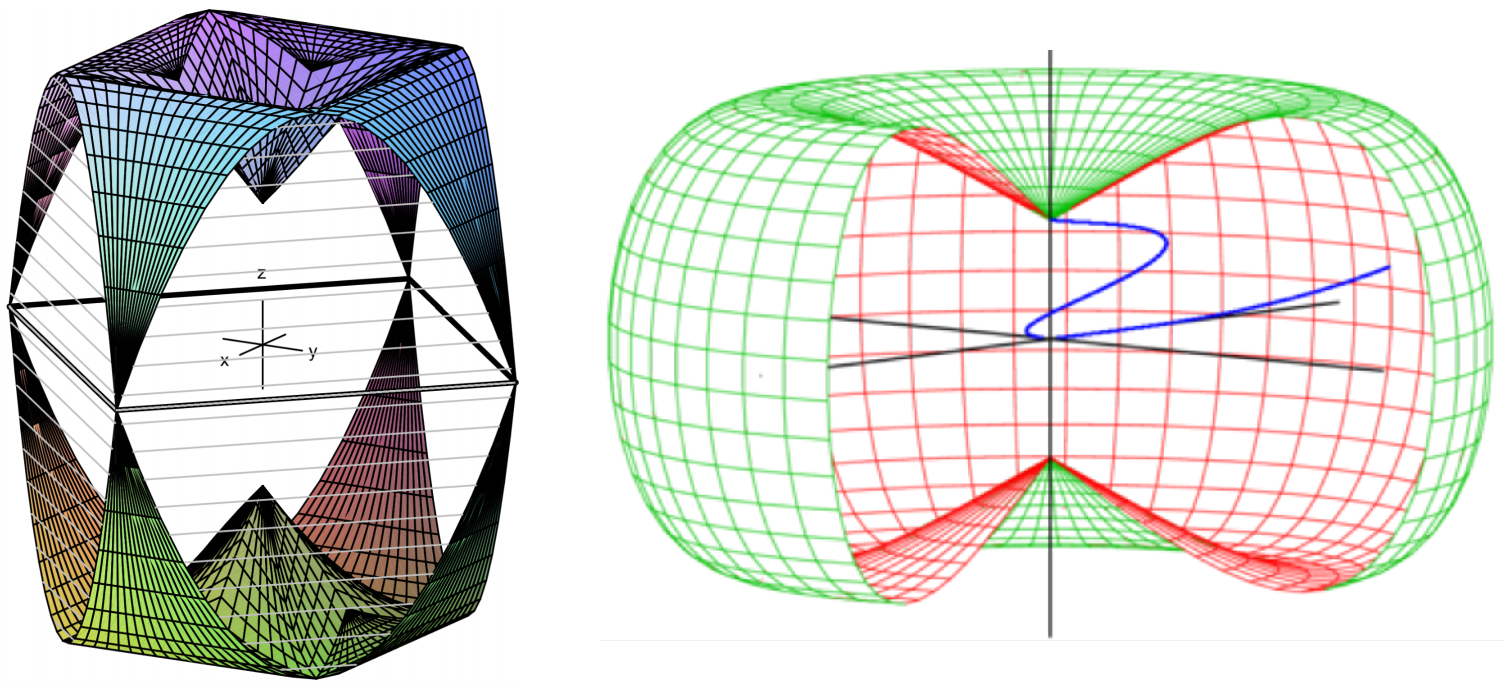}
\caption{Visualization of spheres in $H(\mathbb{Z})$ and $H(\mathbb{R})$. The pictures were taken from \cite{Breulliard,Notes}.}\label{P1}
\end{figure}
The geometry of the discrete Heisenberg group $H(\mathbb{Z})$ is closely related to the geometry of the Heisenberg group $H(\mathbb{R})$, see Figure \ref{P1}. There is a remarkable fact (see \cite{Notes}) that ${\rm Cay}(H(\mathbb{Z}),\{a,b\})$ is quasi-isometric to $H(\mathbb{R})$ and any geodesic curve in $H(\mathbb{R})$ is close to some geodesic in $H(\mathbb{Z})$.

\subsection{Main results}

We prove the following curious fact.

\begin{Proposition}\label{Pr1}
In the discrete Heisenberg group $H(\mathbb{Z})$, any geodesic word is a prefix of a dead end word.
\end{Proposition}

Note that any geodesic word is either prefix of an infinite geodesic word or prefix of a dead end word. We prove that in $H(\mathbb{Z})$ any finite prefix of an infinite geodesic word is a prefix of some dead end word. It is easy to see that the set of dead end words is equal to the set of all geodesic representatives of dead end elements. Here is our first main result.

\begin{Theorem}\label{Th1}
An element $g \in H(\mathbb{Z})$ is a dead end element if and only if $g = [a,b]^k$ for some $k \neq 0$.
\end{Theorem}
This result implies that the set of all dead end words coincide with the set of geodesic representatives of commutators $[a,b]^k$ with $k\neq 0,$ and any geodesic word is a prefix of some such representative.
\begin{figure}[H]
\centering
\includegraphics[width = 16cm]{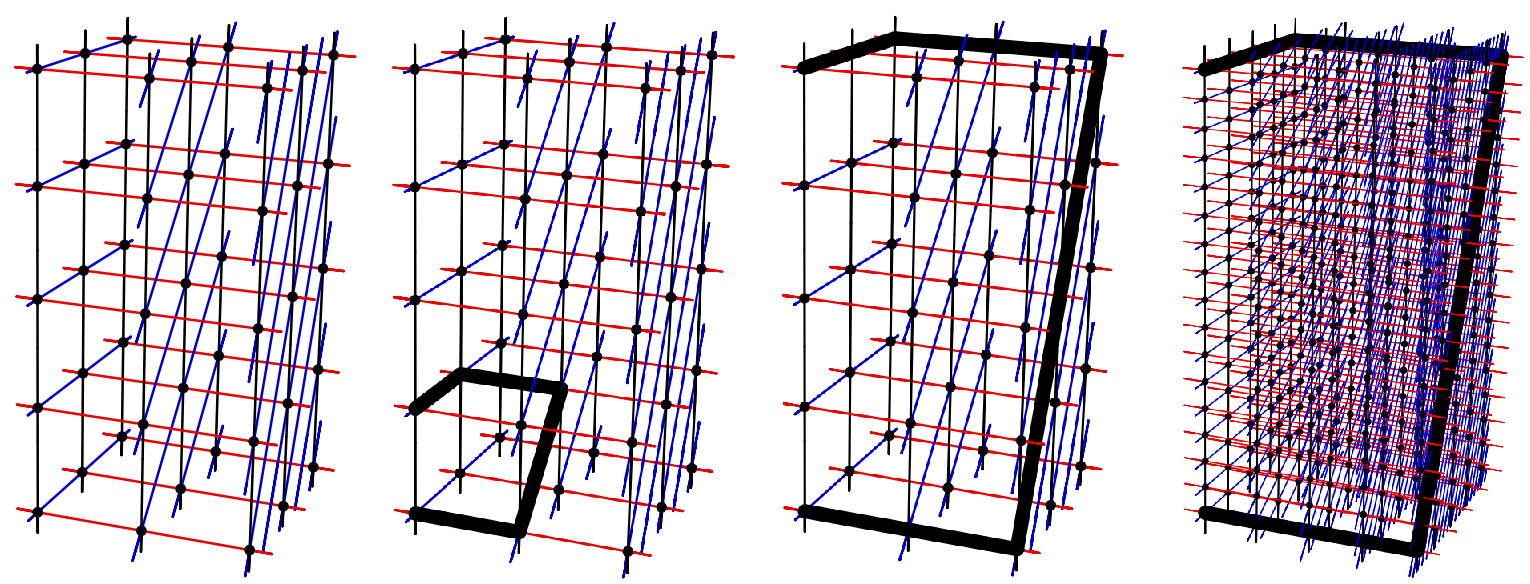}
\caption{The Cayley graph ${\rm Cay}(H(\mathbb{Z}), \{a,b, [a,b]\})$ with respect to the three-element generating set and paths defined by geodesic words in the alphabet $\{a,b\}$.
The pictures were taken from \cite{Notes}.}\label{P2}
\end{figure}
Since $H(\mathbb{Z})$ is a free nilpotent group of rank 2 and class 2, any element of the group has a unique representative of the form $[a,b]^k b^m a^n$. The pair $(n,m)$ and the number $k$ are called the coordinates and the area of $g$, respectively, see below. Note that $[a,b]^k b^m a^n\cdot a^{\pm 1} = [a,b]^k b^m a^{n\pm 1}$ and $[a,b]^k b^m a^n \cdot b = [a,b]^{k\pm n} b^{m\pm 1} a^n$. There is an embedding of the Cayley graph ${\rm Cay}(H(\mathbb{Z}), \{a,b\})$ in $H(\mathbb{R}) \cong \mathbb{R}^3$ given by $[a,b]^k b^m a^n \mapsto (n,m,k) \in \mathbb{R}^3.$ The Cayley graph is a $"$space lattice$"$ consisting of $"$rotating plane lattices$"$. Each of these plane lattices consists of parallel lines. Figure \ref{P2} shows the Cayley graph of $H(\mathbb{Z})$ with respect to the generating set $\{a,b,aba^{-1}b^{-1}\}$.

\begin{figure}[H]
\centering
\includegraphics[width = 14cm]{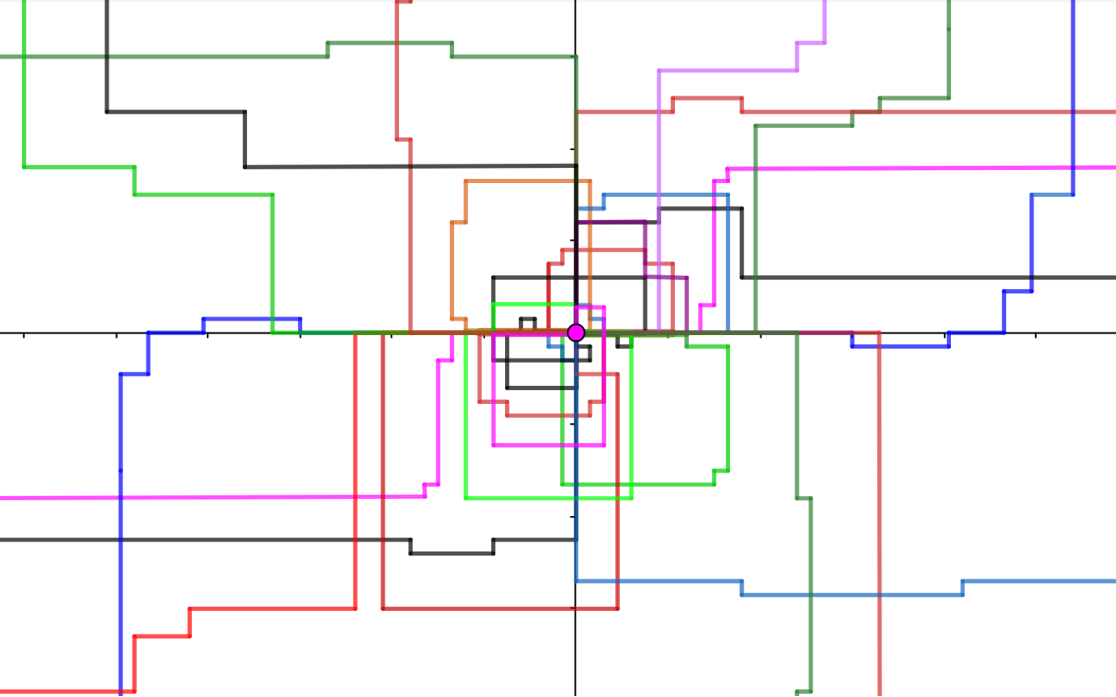}
\caption{The picture illustrates the projection of paths defined by geodesic words.}\label{Plane}
\end{figure}

Consider an orthogonal projection of the Cayley graph of $H(\mathbb{Z})$ on the plane $k=0$. This plane can be identified with the Cayley graph ${\rm Cay}(\mathbb{Z}\times\mathbb{Z}, \{a,b\})$. From an algebraic point of view, one has the following split short exact sequence with the abelianization map given by ${\rm Ab}: [a,b]^k b^m a^n \mapsto b^m a^n$
\begin{center}
\begin{tikzcd}
1\arrow{r} & \mathbb{Z}\arrow{r} & H(\mathbb{Z})\arrow{r}{{\rm Ab}} & \mathbb{Z}\times \mathbb{Z}\arrow{r} & 1
\end{tikzcd}
\end{center}
The is a correspondence between words in the alphabet $\{a,b,a^{-1},b^{-1}\}$ and a paths in ${\rm Cay}(H(\mathbb{Z}), \{a,b\})$ starting at $(0,0,0)$. The result of the projection of a path on the plane is a path, which is defined by the same word in ${\rm Cay}(\mathbb{Z}\times\mathbb{Z}, \{a,b\})$. Given a vertex corresponding to an element $g = [a,b]^k b^m a^n$, its image under the projection is a point $(n,m) \in \mathbb{Z}\times \mathbb{Z}$.

The representation above is said to be the geometric model of $H(\mathbb{Z})$. For details see \cite{VM}. An image of geodesic words in $H(\mathbb{Z})$ under the projection are oriented polygonal chains on the plane, see Figure \ref{Plane}.

A polyomino on the plane $\mathbb{R}^2$ is a connected union of subsets of the form $[n,n+1]\times [m,m+1]$ for $n,m \in \mathbb{Z}$. An area, a boundary, and a perimeter of such polyomino are defined in the natural way. By a polygonal chain we mean a connected union of subsets of the form $\{n\}\times [m,m+1]$, $[n,n+1]\times \{m\}$. A polygonal chain is said to be simple if it is homeomorphic to either the closed interval $[0,1]$ or the circle $S^1$. In the latter case, the simple polygonal chain is said to be closed. For example, the boundary of a polyomino is a union of simple closed polygonal chains. Given an area $k \in \mathbb{N},$ we say that a polyomino with area $k$ is a minimal perimeter polyomino if it has the lowest possible perimeter among all polyomino with area $k$. Any minimal perimeter polyomino with area $k$ is contractible and has perimeter $p(k) := 2 \lceil 2 \sqrt{k} \rceil$, see \cite{K}. Figure \ref{P3} shows a complete list of polyominoes with minimum perimeter $p(k)$ for $k \leq 11$.

\begin{figure}[H]
\centering
\includegraphics[width = 10cm]{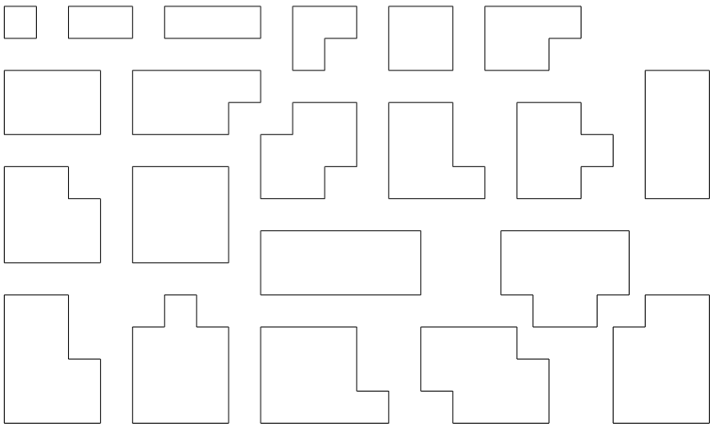}
\caption{Minimal perimeter polyominoes with area $k \leq 11$ up to translations, rotations, and reflections. The picture was taken from \cite{SK}.}\label{P3}
\end{figure}

The is a concept of a (combinatorial) Young tableau that is described e.g. in \cite{YoungTableau}. A polyomino is said to be a Young tableau if is represent a combinatorial Young tableau geometrically. 

Given $k\in \mathbb{N}$ and a rectangular polyomino $a\times b$ with $ab \geq k$ and $2(a+b) = p(k)$, one can delete four Young tableaux of total area $ab-k \geq 0$ from the polyomino's corners and obtain a minimal perimeter polyomino. This process is called deleting squares of degree 2, see Figure \ref{P4}. It turned out that any minimal perimeter polyomino can be obtained in such a way.

\begin{Lemma}[Lemma 2 in \cite{K}]
Each polyomino with minimum perimeter $p(k)$ can be obtained by deleting squares, of degree 2, of a rectangular polyomino with perimeter $p(k)$ consisting of at least $k$ squares.
\end{Lemma}

\begin{figure}[H]
\centering
\includegraphics[width = 17cm]{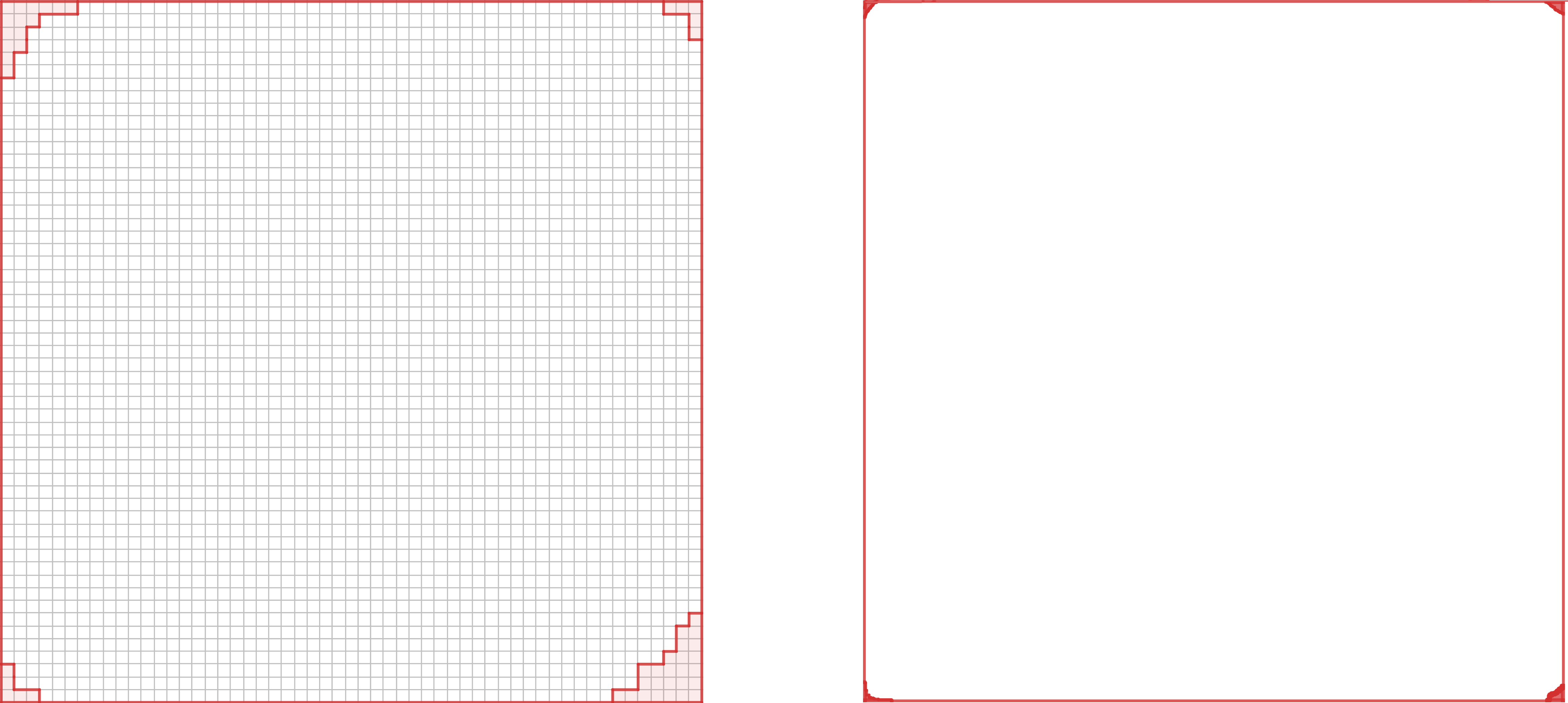}
\caption{Deleting squares of degree 2 process (on the left): the minimal perimeter polyomino is the complement of the four red Young tableaux in the rectangular. The right picture illustrates the fact the total area of the Young tableaux is essentially less than the area of the rectangular as the area goes to infinity.
}\label{P4}
\end{figure}

Given a minimal perimeter polyomino whose the boundary contains the origin, one can choose an orientation on the boundary and obtain so-called oriented polyomino. Starting from the origin, one can read a word $w$ in the alphabet $\{a,b,a^{-1},b^{-1}\}$ along the boundary. Since the boundary is a closed polygonal chain, the word $w$ defines an element of $H(\mathbb{Z})$ of the form $[a,b]^k$. It is easy to see that $|k|$ is an area of the minimal perimeter polyomino. Moreover, the orientation is positive if and only if $k\geq 0$. 
\begin{figure}[H]
\centering
\includegraphics[width = 6cm]{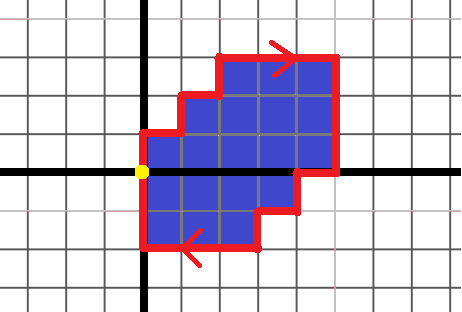}
\caption{An example of a minimal perimeter polyomino on the plane. Its oriented boundary corresponds to the word $w=bababa^3b^{-3}a^{-1}b^{-1}a^{-1}b^{-1}a^{-3}b^2$. The resulting element is $[w] = [a,b]^{-19}$.}\label{PolyDead}
\end{figure}

\begin{Lemma}\label{L0}
The word $w$ defined above is a dead end word.
\end{Lemma}

It turned out that any dead end word can be obtained in such a way:

\begin{Proposition}\label{Pr2}
Given $k \in \mathbb{Z},$ the map described above is a bijection between the set of all geodesic representatives of $[a,b]^k$ and the set of all oriented minimal perimeter polyominoes with area $|k|$ whose the boundary contains the origin.
\end{Proposition}

The proposition implies an explicit grammatical description for the set of all dead end words stated below. We introduce the following notations
\begin{align*}
q_-(k) := \left\lceil \dfrac{\lceil 2 \sqrt{k} \rceil - \sqrt{\lceil 2 \sqrt{k} \rceil^2 - 4k}}{2} \right\rceil, \qquad q_+(k) := \left\lfloor \dfrac{\lceil 2 \sqrt{k} \rceil + \sqrt{\lceil 2 \sqrt{k} \rceil^2 - 4k}}{2} \right\rfloor.
\end{align*}
Consider the following subword transformations of words: $ab \mapsto ba$, $ba^{-1} \mapsto a^{-1}b$, $a^{-1}b^{-1} \mapsto b^{-1}a^{-1}$, and $b^{-1}a \mapsto ab^{-1}$. For example, one can transform $abbba^{-1} \mapsto babba^{-1} \mapsto baba^{-1}b$. These transformations are called swaps of neighboring letters. The main result of this article is the following

\begin{Theorem}\label{Th2}
Suppose $k \geq 0$ and $L \in [q_-(k), q_+(k)]$ are integers, we put $u := a^L b^{\lceil 2 \sqrt{k} \rceil} a^{-L} b^{-\lceil 2 \sqrt{k} \rceil}$. Consider a word $w$ that can be obtained from $u$ by exactly $L(2 \lceil 2 \sqrt{k} \rceil - L) - k$ subword transformations as above; consider all cyclic shifts of $w$ and their inverses. All such words are dead end words and any dead end word is of that form (for some $k,L$). Moreover, any geodesic word is a prefix of some such word.
\end{Theorem}

\subsection{Acknowledgment}
The authors are grateful to A. V. Malyutin for useful discussions.

\section{Terminology and auxiliary results}

We use the terminology of combinatorial and geometric group theory that can be found in \cite{OH, CL}. Given a set of symbols $S$, we write $S^{-1} := \{s^{-1} \mid s \in S\}$. We also write $(s^{-1})^{-1} := s$ and $s^1 := s$. Denote by $F(S)$ the set of all words in alphabet $S \cup S^{-1}$. Given a word $w = s_1^{e_1} \ldots s_n^{e_n}$, we denote by $w^{-1} := s_n^{-e_n} \ldots s_1^{-e_1}$ the inverse word of the word $w$. By a cyclic shift of a word $w = s_0^{e_0} \ldots s_n^{e_n}$ we mean any word of the form $s_k^{e_k}s_{k+1}^{e_{k+1}} \ldots s_{k-1}^{e_{k-1}}$, where indexes are considered to be the residues modulo $n$. By $[u,v]$ we denote the commutator if two words, that is, the word $uvu^{-1}v^{-1}$. Any word $w \in F(S)$ defines a path in the Cayley graph, which ends in a group element denoted by $[w] \in G$. Below we write $S := \{a,b\}$.

Consider the following semigroup (anti) automorphisms of $F(a,b)$, which were defined in \cite{VM}: 
\begin{enumerate}
\item the eight involutions (include identity) induced by the permutations on $\{a,b,a^{-1},b^{-1}\}$ that send pairs of inverse symbols to pairs of inverse symbols;
\item the involution that sends every word to the inverse word;
\item the involution that rearranges the letters of every word in the reverse order.
\end{enumerate}
These transformations are called geometric involutions. In Lemma 2.2. of \cite{VM} it is proved that they send geodesic words to geodesic words. It is now hard not to show that they send dead end words to dead end words. Note that in the plane model of $H(\mathbb{Z})$ the first eight transformations correspond to the action of dihedral group $D_4$ of order 8 on the plane and transformation of the second type correspond to a translation of the polygonal chain's end to the origin and reversing the orientation.

The following explicit formula for the length of elements of $H(\mathbb{Z})$ is given in \cite{SK}. Let $g \in H(\mathbb{Z})$ has coordinates $(n,m)$ and area $k$. Let $0 \leq m \leq n$, $k \geq 0$ and $n > 0$, i.e. $(k,m,n) \neq (0,0,0)$. 
\begin{enumerate}
\item If $n \leq \sqrt{k}$, then $l(g) = 2\lceil 2\sqrt{k}\rceil - n - m$;
\item If $n \geq \sqrt{k}$ and
\begin{enumerate}
    \item $nm \geq k$, then $l(g) = n+m$;
    \item $nm \leq k$, then $l(g) = 2 \lceil k/n \rceil + n - m$.
\end{enumerate}
\end{enumerate}

We also need the explicit description of infinite geodesic words, stated below.

\begin{Lemma}[Theorem 1.1. in \cite{VM}]\label{VMLEMMA}
In the discrete Heisenberg group, the set of all right-infinite geodesic words in the alphabet $S := \{a, b, a^{-1}, b^{-1}\}$ is the union of the following two classes.
\begin{itemize}
\item[Class $\mathcal{A}$] homogeneous words, that is, the words without inverse letters;
\item[Class $\mathcal{B}$] the words of the form $Vxy^mx^{-1}y^{+\infty}$ and the words of the form $x^ny^{-1}x^myVy^{+\infty}$, where $\{x, y\}$ is a pair of elements from $S$ that are not inverse, $n$ is a nonnegative integer, $m$ is a positive integer, and $V$ is a word in the alphabet $\{x, y\}$ that can be brought to the form $x^i y^j$ by less than $m$ swaps of some neighboring letters.
\end{itemize}
\end{Lemma}

\begin{figure}
\centering
\includegraphics[width = 14cm]{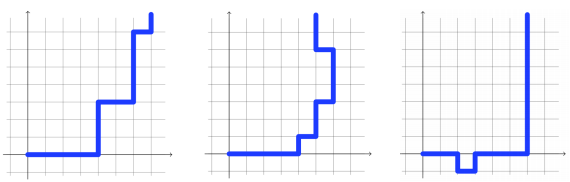}
\caption{Examples of shapes of the infinite geodesic words.}\label{Infinite}
\end{figure}

Also, we need some auxiliary results. The first one states that $[a,b]^k$ are dead end elements for $k \neq 0$, and the second states that geodesic words represent simple polygonal chains.

\begin{Lemma}\label{L1}
Suppose $w \in F(a,b)$ is a geodesic word representing an element of the form $[a,b]^k$ for $k \neq 0$. Then $w$ is a dead end word.
\end{Lemma}

\begin{Lemma}\label{SIMPLE}
Let $w \in F(a,b)$. If $w$ is geodesic in $H(\mathbb{Z})$, then the corresponding polygonal chain in the geometric model of $H(\mathbb{Z})$ is simple.
\end{Lemma}

\section{Proofs of the main results}

\begin{proof}[Proof of Lemma \ref{L0}]
It is easily follows from the explicit length function formula for $H(\mathbb{Z})$ stated above and the fact that $p(k) = 2 \lceil 2 \sqrt{k} \rceil$.
\end{proof}

\begin{proof}[Proof of Lemma \ref{L1}]
Recall that if an element is a dead end element, then any geodesic representative of that element is a dead end word. Note that if for any $s \in S \cup S^{-1}$ there exists a geodesic representative of $g$ whose last letter is $s$, then $g$ is a dead end element. Using this, we prove that $[a,b]^k$ is a dead end for $k \neq 0$. Since the inverse of a dead end element is itself a dead end element, we assume that $k > 0$. Choose a geodesic representative $u$ of $[a,b]^k$ in the form of a boundary of positively-oriented minimal perimeter polyomino with area $k$ whose boundary the contains the origin. Note that $u$ contains all possible letters from $S \cup S^{-1}$. Note also that any cyclic shift of $u$ is a geodesic representative of $[a,b]^k$. As noticed above, this shows that $[a,b]^k$ is a dead end element, and hence any geodesic representative $w$ of $[a,b]^k$ is a dead end word.
\end{proof}

\begin{proof}[Proof of Lemma \ref{SIMPLE}]
Assume the polygonal chain is not simple. Choose a closed polygonal chain lying in it and consider a subword $u$ of $w$ that represents the corresponding shape. Since the polygonal subchain is closed, $[u] = [a,b]^k$ and hence $u$ is a dead end word by Lemma \ref{L1}. Since $u$ is a subword of the geodesic word $w$, then $u=w$, which contradicts the fact that the chain is not simple.
\end{proof}

\begin{proof}[Proof of Theorem \ref{Th1}]
From Lemma \ref{L1} it follows that any element of the form $[a,b]^k$, $k\neq 0$, is a dead end. Now let $g \in H(\mathbb{Z})$ be a dead end element. Suppose $g$ has coordinates $(n,m)$ and area $k$, where $k,m,n \in \mathbb{Z}$. Take a word $w \in F(a,b)$ with a property $[w] = g$ and $|w| = l(g)$. By applying a suitable geometric involution of the first type, one can assume that the coordinates of $g$ satisfy 
$0 \leq m \leq n$. Note that geometric involution of the third type preserves coordinates but reverse the orientation. Using this, one can assume that $k \geq 0$. Firstly, assume that $n \leq \sqrt{k}$. 
The latter is equivalent to $n \leq \lceil \sqrt{k} \rceil$. We have $l(g) = 2\lceil 2\sqrt{k}\rceil - n - m$. We put $h = ga^{-1}$. It is easy to see that $h$ has coordinates $(n-1,m)$ and area $k$. Since $(n-1)m < nm < k$ and $n-1 < n \leq \sqrt{k}$, we have $l(h) = 2\lceil 2\sqrt{k}\rceil - (n-1) - m = l(g)+1 = |wa^{-1}|$, and hence $wa^{-1}$ is a geodesic word. This contradicts the fact that $w$ is a dead end word. 

Now we assume that $n > \sqrt{k}$. The latter is equivalent to $n \geq \lceil \sqrt{k}\rceil $. Consider two cases:
\begin{enumerate}
\item $nm \geq k$. In this case, we have $l(g) = n+m$. Note that $g$ has a positive word representative. Indeed, by dividing with a reminder $k = qm + r$, where $0 \leq r < m$, we see that $u = a^qb^{m-r}ab^ra^{n-q-1}$ has area $k$ and coordinates $(n,m)$. The length of $u$ is $q+m-r+1+r+n-q-1 = m+n$, and hence $u$ is a geodesic representative of $g$. Since all positive words are geodesic, $ua$ and $ub$ are geodesic words too. This contradicts the fact that $g$ is a dead end.
\item $nm < k$. In this case we have
$l(g) =  2 \lceil k/n \rceil + n - m$. Let $h = gb$. It is easy to see that $h$ has coordinates $(n, m+1)$ and area $K := k+n$. Consider two subcases.
\begin{enumerate}
\item $n < \sqrt{n+k}$. In this case $n \leq \sqrt{K},$ and hence we have $l(h) = 2\lceil 2\sqrt{k+n}\rceil - n - (m+1)$. We claim that $l(h) = l(g)+1$. By the triangle inequality, $|l(h) - l(g)| \leq 1.$ Since the relations in $H(\mathbb{Z})$ have even length, $l(h) \neq l(g)$. It remains to prove that $l(h) \geq l(g)$. The latter is equivalent to $2\lceil 2\sqrt{k+n}\rceil - n - m-1>2 \lceil k/n \rceil + n - m$, i.e. $A \geq 0$ for $A := 2\lceil 2\sqrt{k+n}\rceil - \lceil k/n \rceil - 2n - 1$. Since $\lceil 2x\rceil>2y$ for all $x>y$, then $\lceil 2\sqrt{k+n}\rceil > 2n$. Furthermore, from $n\sqrt{k} \geq k$ it follows that $\left\lceil k/n \right\rceil \leq \lceil \sqrt{k} \rceil.$ Finally, from $n > \lceil\sqrt{k}\rceil$ it follows that $n \geq \lceil\sqrt{k}\rceil + 1$, since both of the numbers are integers. Now we have
\begin{align*}
A = 2\lceil 2\sqrt{k+n}\rceil - 2\left\lceil k/n\right\rceil -2n -1 > 4n - 2\lceil \sqrt{k}\rceil - 2n - 1 = 2(n - \lceil \sqrt{k}\rceil) - 1 \geq 2 - 1 \geq -1.
\end{align*}
It follows that $A > -1$ and hence $A \geq 0$.
\item $n \geq \sqrt{n+k}$. In this case $n(m+1) = nm+m < k + m \leq k+n = K$ and $n \geq \sqrt{K}$. We have $l(h) = 2 \lceil (k+n)/n \rceil + n - (m+1) = 2 \lceil k/n \rceil + 2 + n - (m+1) = l(g)+1 = |wb|$. This means that $wb$ is a geodesic word, which contradicts to the fact that $w$ is a dead end word.
\end{enumerate}
\end{enumerate}
The theorem is proved.
\end{proof}

\begin{figure}[H]
\centering
\includegraphics[width = 7cm]{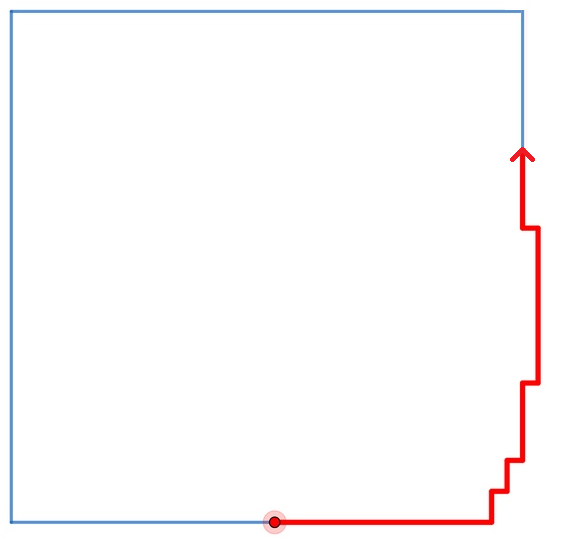}
\caption{Any prefix of an infinite geodesic word from the class $\mathcal{B}$ corresponds to a prefix of a minimal perimeter polyomino boundary.}\label{Prefix}
\end{figure}

\begin{proof}[Proof of Proposition \ref{Pr2}]
It is easy to see that the map is an injective function. We claim that this function is surjective. One has to prove that any geodesic representative $w$ of $[a,b]^k$ corresponds to a boundary of an oriented minimal perimeter polyomino. Without loss of generality, we can assume that $k>0$. It follows from Lemma \ref{SIMPLE} that $w$ represents a simple polygonal chain and hence a boundary of a polyomino $P_1$ with area $|k|$. We claim that $P_1$ is a minimal perimeter polyomino. Consider a minimal perimeter polyomino $P_2$ with area $k$. Note that its boundary represents a word $u$ corresponding to $[a,b]^k$. Since $P_2$ is minimal, $|u| \leq |w|$. Since $w$ is geodesic, $|w| \leq |u|$. It follows that $|w| = |u|$, hence $P_1$ is minimal. The proposition is proved.
\end{proof}

\begin{proof}[Proof of Proposition \ref{Pr1}]
Let $w$ be a geodesic word. If $w$ is a prefix of a dead end word, there is nothing to prove. Suppose $w$ is a prefix of an infinite geodesic word $u$. Due to Lemma \ref{VMLEMMA}, $u$ lies in either the class $\mathcal{A}$ or the class $\mathcal{B}$. By increasing the length of the finite prefix, we can find a minimal perimeter polyomino that contains the prefix in its boundary, see Figure \ref{Prefix}. By Lemma \ref{L0}, the boundary of a minimal perimeter polyomino is a dead end word. Therefore, $w$  is a prefix of a dead end word.
\end{proof}

\begin{proof}[Proof of Theorem \ref{Th2}]
The theorem follows from Proposition \ref{Pr1}, Proposition \ref{Pr2}, and the explicit description of the minimal perimeter polyomino, stated above.
\end{proof}

\end{document}